\documentclass[10pt, a4paper]{amsart}
\usepackage{amssymb,amsfonts,amsmath,amsthm}
\usepackage{enumerate}
\usepackage{leftidx}
\usepackage{tikz}
\usetikzlibrary {arrows,shapes,trees,calc,through,patterns,decorations.markings,decorations.pathreplacing}
\usepackage[numbers]{natbib}

\usepackage{imakeidx}
\usepackage{diagbox}
\usepackage{hyperref}
\usepackage{array}
\makeindex

\newcommand{\Z}{\ensuremath{\mathbb{Z}}}
\newcommand{\Q}{\ensuremath{\mathbb{Q}}}
\newcommand{\R}{\ensuremath{\mathbb{R}}}

\newcommand{\F}{\ensuremath{\mathbb{F}}}

\newcommand{\calC}{\ensuremath{\mathcal{C}}}

\newcommand{\End}{\ensuremath{\mathrm{End}}}
\newcommand{\Jac}{\ensuremath{\mathrm{Jac}}}

\newcommand{\tr}{\ensuremath{\mathrm{tr}}}

\renewcommand{\mod}{\ensuremath{\,\,\mathrm{mod}\,\,}}

\renewcommand{\bar}[1]{\overline{#1}}

\newtheorem{thm}{Theorem}[section]

\newtheorem{prop}[thm]{Proposition}
\newtheorem{lemma}[thm]{Lemma}
\newtheorem{cor}[thm]{Corollary}
\newtheorem{remark}[thm]{Remark}
\newtheorem{exm}[thm]{Example}
\newtheorem{con}[thm]{Conjecture}

\title[Counter-examples to the Hasse Principle]{Counter-examples to the Hasse Principle among the twists of the Klein Quartic}
\author[Lorenzo Garc\'ia]{Elisa Lorenzo Garc\'ia}

\address{Elisa Lorenzo Garc\'ia: Institut de Math\'ematiques, Universit\'e de Neuch\^atel, Rue Emile-Argand 11, 2000, Neuch\^atel, Switzerland --- Laboratoire IRMAR, Office 602,
Universit\'e de Rennes 1, Campus de Beaulieu, 35042, Rennes Cedex, France} \email{elisa.lorenzo@unine.ch, elisa.lorenzogarcia@univ-rennes1.fr}
\address{Micha\"el Vullers:}
\email{mtivullers@gmail.com}
\author[Vullers]{Micha\"el Vullers \vspace{.5cm} \\ {\textit{D\MakeLowercase{edicated to }J\MakeLowercase{aap} T\MakeLowercase{op on the occasion of his 62th birthday.}}}}

\begin{document}
\begin{abstract}
    In this paper we inspect from closer the local and global points of the twists of the Klein quartic. For the local ones we use geometric arguments, while for the global ones we strongly use the modular interpretation of the twists. The main result is providing families with (conjecturally infinitely many) twists of the Klein quartic that at counter-examples to the Hasse Principle.
\end{abstract}
\maketitle
\section{Introduction}

The Klein quartic $\calC$ is the curve in the projective plane given by the equation $$X^3Y+Y^3Z+Z^3X=0$$ which is isomorphic to the modular curve $X(7)$ over the cyclotomic field $\mathbb{Q}(\zeta_7)$. 
It has a number of interesting properties which are deeply studied in \citep{Elkies98}. For instance it is a Hurwitz curve, meaning that it has the maximum number of automorphisms possible for its genus. In fact it is the smallest genus example of such a curve. 

Given a curve $C$ over a field $k$, we say that a curve $C'/k$ is a twist of $C$ if over an algebraic closure $\bar k$ of $k$ the curves $C$ and $C'$ are isomorphic. Geometrically speaking both curves are equal, however they may have very different arithmetic properties.
In this paper we focus our attention on the following question: Suppose that $C'$ is a twist of $\calC$,  \emph{does $C'$ violate the Hasse principle?} 
That is, does $C'$ have a point over every completion of $\Q$ but not over $\Q$ itself?

In \citep{2009arXiv0911.4537O} and \citep{2012arXiv1205.3424O} a similar question was asked and answered for the modular curves $X_0(N)$. The modularity of these curves is strongly used there to prove the existence of local and global points. In this paper we start by using a classification of all the twists of the Klein quartic \cite{Elisa2}, we use geometric arguments as well as Hensel's Lemma to deal with the existence of local points and we only use modular reasonings to study rational points. Those differ by the ones used in the aforementioned papers. Ther are based on the central result in \cite{JFern}. 

Our main result is the following:

\begin{thm}\label{thm:main}
    There are (conjecturally infinetely many) twists of the Klein quartic yielding counterexamples to the Hasse principle.
\end{thm}

\section{Non-Archimedean local points}
In this section we study the existence of $\mathbb{Q}_p$-points for primes $p$ not dividing the discriminant or the order of the automorphism group of the twists of the Klein quartic.
\begin{prop}
Let $p$ be a prime different from $2$, $3$ or $7$ and let $C'/\F_p$ be a twist of the Klein quartic such that $p$ does not divide its discriminant. Then $C'(\F_p)\neq \emptyset$.
\label{local-points}
\end{prop}
\begin{proof}
If $p\geq 33$, the result follow immediately from the Hasse-Weil bound, so assume $p<33$. 
Let $\pi$ denote the Frobenius on $\Jac(C')$. Since $$\# C'(\F_p)=p+1-\tr(\pi),$$ it suffices to show that $\tr(\pi)\neq p+1$. 
In their paper \citep{Meagher2010} Meagher and Top show that
$$\tr(\pi)=\left\{\begin{array}{ll} 0 &\text{if } p\equiv 3,5,6 \mod 7\\ 3\tr(\pi_E), 0 ,-\tr(\pi_E) ,\tr(\alpha\pi_E)\text{ or }\tr(\pi_E)&\text{if }p\equiv 1,2,4\mod 7\end{array}\right.$$
where $\pi_E$ is the Frobenius of the elliptic curve $E/\F_p$ given by $y^2+xy = x^3+5x^2+7x$ and $\alpha\in\End(E)$ satisfies $\alpha^2+\alpha+2=0$. 
%Moreover they show that the case $\tr(\pi)=3\tr(\pi_E)$ only occurs if $p\equiv \pm 1\mod 7$.

It follows immediately that if $p\equiv 3,5,6\mod 7$ then there is nothing to prove, so assume $p\equiv 1,2,4\mod 7$. This gives us only the cases $p=11$, $23$ or $29$. For which the elliptic curve $E$ is ordinary and  $\End(E)\cong \Z[\alpha]\subset \Q(\sqrt{-7})$ after making $\alpha=\frac{-1+\sqrt{-7}}{2}$. Moreover, after this identification, we get $\pi_E= \pm2\pm\sqrt{-7}$, $\pm4\pm\sqrt{-7}$, $\pm1\pm2\sqrt{-7}$ for $p=11$, $23$, $29$ respectively.

% Consider $N_i=\alpha_i\pi_E\in\End(E)$. Note that $\End(E)\cong \Z[\alpha_i]\subset \Q(\sqrt{-7})$ and that $\alpha_i$ and $\pi_E$ satisfy 
% \begin{eqnarray*}
% \alpha_i^2+\alpha_i +2 &=& 0,\\ \pi_E^2-\tr(\pi_E)\pi_E + p &=& 0,
% \end{eqnarray*}
% therefore $\alpha_i$ and $N_i$ satisfy 
% \begin{eqnarray*}
% \alpha_i^2+\alpha_i +2 &=& 0,\\ N_i^2 -\tr(\pi_E)N_i\alpha_i + p\alpha_i^2 &=& 0.
% \end{eqnarray*}
% In order to find $\tr(\alpha_i\pi_E)=\tr(N_i)$ we take the resultant with respect to $\alpha_i$ which yields the following polynomial: 
% $$N_i^4+\tr(\pi_E)N_i^3 + (2\tr(\pi_E)-3p)N_i^2 + 2p\tr(\pi_E)N_i + 4p^2.$$
% Note that the minimal polynomial of $N_i$ is a degree $2$ divisor of this polynomial, moreover $N(N_i)=N(\alpha_i)N(\pi_E) = 2p$ hence the minimal polynomial of $N_i$ is one of the following divisors
% $$N_i^2 + \frac{\tr(\pi_E) + \sqrt{-7\tr(\pi_e)^2 +28p}}{2} N_i + 2p,\quad N_i^2 + \frac{\tr(\pi_E) - \sqrt{-7\tr(\pi_e)^2 +28p}}{2} N_i + 2p.$$
% Since $\End(E)\otimes\Q \cong \Q(\sqrt{-7})$ one has $\tr(\pi_E)^2 - 4p = -7x^2$ for some $x$ which simplifies the above to
% $$N_i^2 + \frac{\tr(\pi_E)+7x}{2}N_i +2p,\quad N_i^2 + \frac{\tr(\pi_E)-7x}{2}N_i +2p.$$
% This yields $$\tr(N_i)=\tr(\alpha_i\pi_E)\in\left\{\frac{\tr(\pi_E) + 7x}{2},\frac{\tr(\pi_E)-7x}{2}\right\},$$
% hence it only depends on $\tr(\pi_E)$. 
% Note that the only cases for $p<33$, $p\neq 2,3,7$ and $p\equiv 1,2,4\mod 7$ are $11$, $23$ and $29$ with $\tr(\pi_E)$ equal to $4$, $8$ and $2$ respectively. 
Straightforward substitution now yields $\tr(\pi)\neq p+1$ hence $C(\F_p)\neq \emptyset$.
\end{proof}

\begin{cor}
Let $p\neq 2,3,7$ be a prime. If $C'$ is a twist of the Klein quartic with $p$ not dividing its discriminant, then $C'(\Q_p)\neq \emptyset$.
\end{cor}
\begin{proof}
Let $\bar{C'}/\F_p$ be the reduction of $C'$, then $\bar{C'}$ is a twist of the Klein quartic over $\F_p$, in particular it is a smooth curve. 
From Proposition \ref{local-points} one has $\bar{C'}(\F_p)\neq\emptyset$ hence by Hensel's Lemma it follows that $C'(\Q_p)\neq \emptyset$.
\end{proof}
 
\section{Real points}

In this section, and before proving that all twists $C'$ of the Klein quartic have real points, we introduce a special twist:

\small{
$$
C_0 : x^4 +y^4+z^4 + 6(xy^3 +yz^3 +zx^3) -3 (x^2y^2+y^2z^2+z^2x^2)+3xyz(x+y+z) = 0.
$$}
\normalsize

We fix the isomorphism $\phi_0:\,C_0\rightarrow\mathcal{C}$ as in \cite[Sec. 6]{Elisa2}. The special property of this twist is that among all twists of the Klein quartic, the Klein quartic itself included, it is the one with the smallest field of definition of its automorphism group, namely $\mathbb{Q}(\sqrt{-7})$. This is why, in \cite{Elisa2} instead of directly computing the twists of $\mathcal{C}$, the twists of $C_0$ are computed. Of course, both sets are equal, but the algorithm developed in \cite{Elisa1} runs  faster while inputing $C_0$. Moreover, $C_0$ is isomorphic over $\mathbb{Q}$ to $X(7)$.

\begin{prop}
Let $C'$ be a twist of the Klein quartic $\calC$. Then $C'(\R)\neq \emptyset$.
\end{prop}
\begin{proof}
By the classification of twists of the Klein quartic in \cite{Elisa2} when $K=\R$ we get that there are only two twists of the Klein quartic up to $\mathbb{R}$-isomorphism, the Klein quartic itself and the twist $C_0$. The Klein quartic has for instance the rational point $(1:0:0)$. Plugging $(x:1:1)$ into $C_0$ yields a degree $4$ polynomial in $x$ over $\R$ which has real solutions, hence $C_0(\R)\neq\emptyset$. 
\end{proof}

\section{Rational points}
In order to study the rational points of the twists of the Klein quartic, we start by recalling some definitions and rephrasing the main theorem of \citep{JFern}.

Let $\chi_p$ be the quadratic character of $G_\Q$ obtained from the $p$-th cyclotomic character.
The Galois action on the $p$-torsion points of an elliptic curve $E/\Q$ produces a representation
$\bar{\rho}_{E,p}:\, G_k\rightarrow\operatorname{PGL}_2(\F_p)$,
determined up to conjugation by the isomorphism class of $E$ unless the $j$-invariant
of $E$ is $0$ or $1728$, and with determinant $\chi_p$. Going the other way, we say that a
representation $\rho:\, G_\Q\rightarrow\operatorname{PGL}_2(\F_p)$
is realized by an elliptic curve $E/\Q$ if $\bar{\rho}_{E,p} =\rho$, where this last equality is considered
up to conjugation inside $\operatorname{PGL}_2(\F_p)$.

The automorphism group of $X(p)$ is naturally (via the modular interpretation) isomorphic to $\operatorname{PSL}_2(\F_p)$ and it is defined over the only quadratic field $k_p$ contained in the cyclotomic field  $\Q(\zeta_p)$. Fix $V:=\begin{pmatrix}0 & -v\\1 & 0\end{pmatrix}$ for $v$ a non-square in $\F_p^\times$. Then, by \cite[Prop. 2.1]{JFern} we have that the action of the non-trivial element $\tau$ of $\operatorname{Gal}(k_p/\Q)$ on an element $g\in \operatorname{Aut}(X(p))\simeq\operatorname{PSL}_2(\F_p)$ is given by $g\mapsto VgV$.

A cocycle $\xi\in\operatorname{H}^1(G_\Q,\operatorname{Aut}(X(p)))$ defining a twist of $X(p)$ can be seen as a representation $\rho:\,G_\Q\rightarrow\operatorname{Aut}(X(p))\rtimes\operatorname{Gal}(k_p/\Q)\simeq\operatorname{PGL}_2(\F_p):\,\sigma\mapsto(\xi(\sigma),\bar{\sigma})$, where the {last isomorphism} is given by sending $(g,1)$ to $g$ and $(1,\tau)$ to $V$. Going the other way around, starting by a representation $\rho:\,G_\Q\rightarrow\operatorname{PGL}_2(\F_p)$ whose determinant is the quadratic character $\chi_p$, we can define the cocycle $\xi$ given by the first projection of $\rho$ by using the previous isomorphism, i.e. setting $\xi(\sigma)=\pi_1(\rho(\sigma))$. This cocycle produces a twist $X_\rho(p)$. We can as well define a cocycle $\xi'(\sigma)=\pi_1(\sigma^{-1})^t$ defining a twist $X'_\rho(p)$. These two twists, defined over the same field, are equivalent under the conditions in \cite[Lemma 2.1]{JFern}.

\begin{thm}(\cite[Thm. 2.1]{JFern})
The map $X(p)\rightarrow X(1)$ define a surjective map from the set of non-cuspidal rational points with $j\neq0,1728$ on the curves $X_\rho(p)$ and $X'_\rho(p)$ to the set of isomorphism classes of elliptic curves defined over $\Q$ realizing $\rho$. 
\label{rationalpoints}
\end{thm}

It is important to note that not every twist of $X(7)$ arises as $X_E(7):=X_{\bar{\rho}_{E,7}}(7)$ or $X_E^-(7):=X'_{\bar{\rho}_{E,7}}(7)$ for some elliptic curve $E$, in particular $\mathcal{C}$ itself does not arise from an elliptic curve. This type of twists have brought the attention of the community: explicit equations for $X_E(7)$ are given in \cite{halberstadt2003}, and the study of their rational points for some particular examples are in \cite{Stoll2005}, which is used to solve a particularly difficult diophantine equation.

In the case in which the elliptic curve $E$ does not have CM, we follow \cite[Thm. 1.5]{Zywina2015} to determine the possible images of the representations $\bar{\rho}_{E,7}$.

\begin{thm}\label{thm:possibilities} Let $E$ be an elliptic curve defined over $\Q$. If $E$ does not have CM, then the image of $\bar{\rho}_{E,7}$ is isomorphic to $S_3$ (then $j(E)=3^3\cdot5\cdot7^5/2^7$) or has cardinality $14$, $42$ or $336$. If $E$ has CM then $\bar{\rho}_{E,7}$ has cardinality $7$ or larger than $12$.
\end{thm}

\begin{proof} The aforementioned reference classify the possibilities for $\rho_{E,7}:\,G_\Q\rightarrow\operatorname{GL}_2(\F_7)$ when $E/\Q$ does not have CM. 

There are only a finite number of $E/\Q$ with CM up to $\bar{\Q}-$isomorphism. When $7$ does not divide the discriminant of the CM-field and up to quadratic twists, the corresponding representation is maximal according to the LMFDB \cite{lmfdb}. When $7$ divides the discriminant then the CM-field is $\Q(\sqrt{-7})$ and again according to the LMFDB, and up to quadratic twist, $\bar{\rho}_{E,7}$ has cardinality $14$.
\end{proof}

Twists not arising from elliptic curves do not have rational points apart from eventually the cusps.

\begin{prop}
Let $\phi\colon C'\rightarrow C_0$ be a twist defined over a field $L$. If $L\cap \Q(\zeta_7)=\Q(\sqrt{-7})$, then $C'$ has no rational cusps. 
\label{cusppoints}
\end{prop}

\begin{proof}
The cusps of $C_0$ are defined over $\Q(\zeta_7)$ but not over $\Q(\sqrt{-7})$. An isomorphism $\phi\colon C'\rightarrow C_0$ maps $C'(\Q)$ to $C_0(L)$. 
If $C_0$ has no cusps over $L$, then $C'$ has no cusps over $\Q$. 
\end{proof}

\section{Counter-examples to the Hasse Principle}
In the classification of twists of the Klein quartic in \cite{Elisa2} there are $11$ different families (or cases) of twists accordingly to the isomorphism class of the Galois group of the field of definition of the twists. The families 6. and 10. do not appear when the Klein quartic is considered to be defined over $\Q$. And the family $11.$ contains all the twists $X_E(7)$ when $\bar{\rho}_{E,7}$ is a big as possible which is the generic case: all those twists have then rational points according to Theorem \ref{rationalpoints}. We hence study the other cases in order to find counter-examples to the Hasse Principle.

\begin{prop}
    Twists from cases $1$. and $2$. satisfy $C'(\mathbb{Q}_2)=\emptyset$
\end{prop}
\begin{proof}
        Case $1$. corresponds to the curve $C_0$. It is straightforward to check that $\bar{C_0}(\F_2)=\emptyset$. For case $2$. we distinguish the cases $m$ even or odd, where $m$ is the parameter of the family given in \cite[Thm. 6.1]{Elisa2}. This time $\bar{{C}'}(\F_2)\neq\emptyset$, but  $\bar{{C}'}(\Z/4\Z)=\emptyset$.
%         g:=-49*x^4 - m^2*y^4 + 31*m^2*z^4 + 210*m*x^2*y*z - 12*m^2*y^3*z
% -66*m^2*y*z^3 - 63*m*x^2*y^2 + 51*m^2*y^2*z^2 - 126*m*x^2*z^2; 
\end{proof}

\begin{prop}
    Twists in case $7$. satisfy $C'(\mathbb{Q}_7)=\emptyset$.
\end{prop}

\begin{proof}
    Independently of the valuations at $7$ of the parameters defining this family we always obtain, after simplification, an equation of the form $Q^2+7G=0$ with $Q,G\in\mathbb{Z}[x,y,z]$ homogeneous of degree 2 and 4 respectively, and not divisible by $7$. This model of the twists clearly shows the hyperelliptic reduction behavior of the Klein quartic at $7$, see \cite[Prop. 1.2]{Elisa-Ritzenthaler}. While trying to construct $\mathbb{Q}_7$-points we find an obstruction. We do find points over $\mathbb{F}_7$ in the conic, but they are double points of the quartic and they cannot even be lifted to $\mathbb{Z}/7^2\mathbb{Z}$.
\end{proof}

\begin{prop}
    For infinitely many twists in cases $5$., $9$. and $11$. we have $C'(\mathbb{Q})\neq\emptyset$.
\end{prop}

\begin{proof}
      They are infinitely many non-CM elliptic curves giving Galois representations with  $\bar{\rho}_{E,7}$ isomorphic to the corresponding groups of order $14$, $42$ and $336$ respectively, see \cite[Thm. 1.5]{Zywina2015}.
\end{proof}

We look next to families that do provide counter-examples to the Hasse Principle.

\subsection{Cases 3., 4. and 8}\label{sec:good}

%Trying to save it
% R<x>:=PolynomialRing(Rationals());
% f:=x^3 - x^2 - 2*x + 1;
% K<a>:=NumberField(f);
% R<x>:=PolynomialRing(K);
% f:=x^3 - x^2 - 2*x + 1;
% Factorization(f);
% b:= a^2 - a - 1;
% c:= - a^2 + 2;
% R<x,y,z>:=PolynomialRing(K,3);
% X:=a*x+b*y+c*z;
% Y:=b*x+c*y+a*z;
% Z:=c*x+a*y+b*z;
% X^3*Y+Y^3*Z+Z^3*X;
% K:=GF(7);
% R<x,y,z>:=PolynomialRing(K,3);
% F:=-x^4 + 10*x^3*y + 10*x^3*z - 27*x^2*y^2 + 9*x^2*y*z - 27*x^2*z^2 + 10*x*y^3 +
%     9*x*y^2*z + 9*x*y*z^2 + 10*x*z^3 - y^4 + 10*y^3*z - 27*y^2*z^2 + 10*y*z^3 -
%     z^4;
% Factorization(F);

% R<x>:=PolynomialRing(Rationals());
% f:=x^3 - x^2 - 2*x + 1;
% K<a>:=NumberField(f);
% R<x>:=PolynomialRing(K);
% f:=x^3 - x^2 - 2*x + 1;
% Factorization(f);
% b:= a^2 - a - 1;
% c:= - a^2 + 2;
% R<x,y,z>:=PolynomialRing(K,3);
% x:=x-y-z;
% F:=-x^4 + 10*x^3*y + 10*x^3*z - 27*x^2*y^2 + 9*x^2*y*z - 27*x^2*z^2 + 10*x*y^3 +
%     9*x*y^2*z + 9*x*y*z^2 + 10*x*z^3 - y^4 + 10*y^3*z - 27*y^2*z^2 + 10*y*z^3 -
%     z^4;F;
% x:=7*x;
% F:=-x^4 + 14*x^3*y + 14*x^3*z - 63*x^2*y^2 - 63*x^2*y*z - 63*x^2*z^2 + 98*x*y^3 +
%     147*x*y^2*z + 147*x*y*z^2 + 98*x*z^3 - 49*y^4 - 98*y^3*z - 147*y^2*z^2 -
%     98*y*z^3 - 49*z^4;
% F/49;
% K:=GF(7);
% R<x,y,z>:=PolynomialRing(K,3);
% F:=-49*x^4 + 294*x^3*y + 294*x^3*z - 651*x^2*y^2 - 1239*x^2*y*z - 651*x^2*z^2 +
%     630*x*y^3 + 1743*x*y^2*z + 1743*x*y*z^2 + 630*x*z^3 - 225*y^4 - 814*y^3*z -
%     1179*y^2*z^2 - 814*y*z^3 - 225*z^4;
% Factorization(F);

For those twists the field of definition is $L=K(\alpha,\beta,\gamma)$ with $K=\Q(\sqrt{-7})$ and $\alpha,\beta,\gamma$ are the roots of a degree $3$ irreducible polynomial $f=x^3+Ax^2+Bx+C$ with coefficients in $\Q$. For case 3 we have that the splitting field of $f$ over $\Q$ has discriminant $\Delta = -7q^2$, for case 4. $\Delta = q^2$, and for case 8 we have that the splitting field of $f$ has discriminant $\Delta$ not of the form $q^2$ or $-7q^2$. 

Let $\phi\colon C'\rightarrow C_0$ be the twist associated with $L$. Then $\phi$ is given by 
$$\left(\begin{array}{ccc}
 d & -3\alpha + 2\beta+\gamma & \alpha\beta -3\beta\gamma+2\alpha\gamma \\
 d & \alpha -3\beta+2\gamma   & 2\alpha\beta +\beta\gamma-3\alpha\gamma \\
 d & 2\alpha +\beta-3\gamma   & -3\alpha\beta +2\beta\gamma+\alpha\gamma \\
\end{array}\right)$$
which has determinant equal to $-21d\sqrt\Delta$. Taking $(d,\Delta)=(\sqrt{-7},-7q^2)$ gives us case 3, $(d,\Delta)=(1,q^2)$ gives us case 4, and taking $(d,\Delta)=(\sqrt{\Delta},\Delta)$ yields case 8. 

Therefore if $p$ is a prime such that $p\neq 2,3,7$ and $p\nmid \Delta$, then $p\nmid \det(M)$ hence $C'$ has good reduction at $p$. 
From Proposition \ref{local-points}, and its corollaries, it follows that for such primes $C'(\Q_p)\neq\emptyset$. 

We can assume the coefficient $A=0$. Hence, accordingly to \cite{Elisa2} the twist $C'$ is given by the equation 
\begin{align*}
F= & \,\,  3d^4x^4-63 d^2B x^2y^2+ 9\cdot 21 d^2C x^2yz +21 d^2 B^2 x^2z^2 +147 d\sqrt\Delta xy^3\\
&+147 d\sqrt\Delta Bxyz^2 - 147d\sqrt\Delta Cxz^3-9\cdot 49B^2 y^4 +98\cdot 27 BC y^3z\\
&+147(2B^3 -27C^2)y^2z^2 -9\cdot 98 B^2C yz^3- 49B^4 z^4 =0.
\end{align*}
With $\Delta = -4 B^3 -27 C^2$.

% R<x,y,z>:=PolynomialRing(Rationals(),3);
% B:=3;
% C:=5;
% D:=-4*B^3-27*C^2;
% Factorization(D);
% F:=3*D^2*x^4-63*D*B* x^2*y^2+ 9*21*D*C* x^2*y*z +21*D*B^2* x^2*z^2 +147* 
% D*x*y^3+147* D* B*x*y*z^2 - 147*D*C*x*z^3-9*49*B^2* y^4 +98*27* B*C* y^3*z
% +147*(2*B^3 -27*C^2)*y^2*z^2 -9*98* B^2*C *y*z^3- 49*B^4* z^4;
% DixmierOhnoInvariants(F);
% Factorization(73101153029318001);
% f:=3^4*7*29*(x^3*y+y^3*z+z^3*x);
% DixmierOhnoInvariants(f);

We study now the $\Q_2$-points, $\Q_3$-points and $\Q_7$-points.

\begin{lemma}
Let $f(x)=x^3+Bx+C$ and $C'$ as above, if $2\nmid C$ then $C'(\Q_2)\neq\emptyset$.
\end{lemma}
\begin{proof}
If $2\nmid C$ then $2\nmid \Delta$ and $F$ reduces modulo $2$ as follows
\begin{eqnarray*}
x^4 +x^2yz+xy^3 +xz^3 +y^2z^2 &\text{if }2\mid B\\
x^4 +x^2y^2+x^2yz +x^2z^2+xy^3 +xyz^2+xz^3 +y^4 +y^2z^2+z^4 &\text{if }2\nmid B\\
\end{eqnarray*}
Note that if $2\mid B$, then the point $(0:1:0)$ is a solution over $\F_2$, If $2\nmid B$, then $(1:1:1)$ is a solution over $\F_2$. Both of which can be lifted with Hensel's Lemma since the derivative with respect to $x$ does not vanish.
\end{proof}

\begin{lemma}
Let $f(x) =x^3+Bx+C$ and $C'$ as above. If  $3\nmid B$
then $C'(\Q_3)\neq\emptyset$.
\end{lemma}
\begin{proof}
If $3\nmid B$, then since $\Delta=-4B^3-27C^2$ it follows that $3\nmid \Delta$ and therefore $3\nmid d\sqrt{\Delta}$ and $3\nmid d$. After making the change of variable $z\mapsto 3z$ and dividing by $3$ the equation, we obtain modulo $3$ the non-singular point $(0:1:0)$, that lifts to a point in $\Q_3$ by Hensel's Lemma. 
\end{proof}

\begin{prop}
Let $f(x)=x^3+Bx+C$ and $C'$ as above. Assume $7\mid B$ and $7\mid\mid C$. Then $C'(\Q_7)\neq \emptyset$.
\end{prop}
\begin{proof}
We check with Magma \cite{magma} that after writing $B=7b$ and $C=7c$ we get non-singular points over $\mathbb{F}_7$ for all the values of $b,c\in\mathbb{F}_7$ unless $c=0$.
% K:=Rationals();
% S<b,c>:=FunctionField(K,2);
% P<x,y,z>:=ProjectiveSpace(S,2);
% B:=7*b; C:=7*c;
% D:=-4*B^3-27*C^2;
% F:=3*D^2*x^4-63*D*B* x^2*y^2+ 9*21*D*C* x^2*y*z +21*D*B^2* x^2*z^2 +147*
% D*x*y^3+147* D* B*x*y*z^2 - 147*D*C*x*z^3-9*49*B^2* y^4 +98*27* B*C* y^3*z
% +147*(2*B^3 -27*C^2)*y^2*z^2 -9*98* B^2*C *y*z^3- 49*B^4* z^4;
% F/7^4;
% K:=GF(7);
% S<b,c>:=FunctionField(K,2);
% P<x,y,z>:=PolynomialRing(S,3);
% F:=(2352*b^6 + 4536*b^3*c^2 + 2187*c^4)*x^4 + (252*b^4 + 243*b*c^2)*x^2*y^2 +
%     (-84*b^3 - 81*c^2)*x*y^3 - 9*b^2*y^4 + (-756*b^3*c - 729*c^3)*x^2*y*z +
%     54*b*c*y^3*z + (-588*b^5 - 567*b^2*c^2)*x^2*z^2 + (-588*b^4 -
%     567*b*c^2)*x*y*z^2 + (42*b^3 - 81*c^2)*y^2*z^2 + (588*b^3*c + 567*c^3)*x*z^3
%     - 126*b^2*c*y*z^3 - 49*b^4*z^4;F;
% Factorization(F);
% P<x,y,z>:=ProjectiveSpace(K,2);
% b:=3; c:=2;
% F:=3*c^4*x^4 + 5*b*c^2*x^2*y^2 + 6*c^3*x^2*y*z + 3*c^2*x*y^3 + 5*b^2*y^4 +
%     5*b*c*y^3*z + 3*c^2*y^2*z^2;
% C := Curve(P,F);
% Points(C);
% SingularPoints(C);

\end{proof}

 For primes $p$ dividing the discriminant we prove the existence of $\Q_p$-points under the following conditions.
 
\begin{prop}
Let $p\neq 2,3,7$ be a prime number. If $p\mid B$ and $p\mid\mid C$ then $C'(\Q_p)\neq \emptyset$.
\end{prop}
\begin{proof}
One can check with Magma \cite{magma} that writing $B=pb$ and $C=pc$ we get that $F\equiv y^2(-3969c^2xy - 441b^2y^2 + 2646bcyz - 3969c^2z^2)\text{ mod }p$ which contains non-singular point over $\mathbb{F}_p$ if $c\neq 0$.
\end{proof}

Let us address now the question about rational points. 

\begin{prop} In cases 4. and 8. $C'(\Q)=\emptyset$.\label{Qptcase4}\end{prop}
\begin{proof}
From Theorem \ref{thm:possibilities} we know that these twists do not come from elliptic curves, i.e. they are not of the form $X_E(7)$. Then  Theorem \ref{rationalpoints} tells us that those twists have no non-cuspidal rational point. Finally, Proposition \ref{cusppoints} gives us the non existence of cuspidal rational points.
\end{proof}

\begin{prop}
In case 3., if $C'\not\simeq_\Q \mathcal{C}$ and the field of definition is not the field $6.0.214375.1$ (with the \cite{lmfdb} notation) then $C'(\Q)=\emptyset$.
\end{prop}
\begin{proof}
%g:=x^6 + 19194*x^4 + 141577317*x^2 + 244209993727;
% raiz de -7 y x^3 + 38382*x^2 + 566155728*x + 1952547484824
% same cubic field x^3 + 383145*x - 401763985
The same proof than in previous proposition applies except that we have to remove the special case in which the elliptic curve has $j$-invariant $3^3\cdot5\cdot7^5/2^7$, in this case the 7-torsion is defined over the field $6.0.214375.1$ or a quadratic extension of it and we need to remove the single case in which the conditions of Proposition \ref{cusppoints} are not hold, namely for the Klein quartic that it does have rational cuspidal points.
\end{proof}

% _<x>:=PolynomialRing(Rationals());
% f:= x^(12) - 5 *x^(11) + 13 *x^(10) - 40 *x^9 + 100 *x^8 - 195*x^7 + 370* x^6 - 590* x^5 + 855 *x^4 - 870* x^3 + 908* x^2 - 400* x + 64;
% K:=SplittingField(f);
% Gal, _, Map := AutomorphismGroup(K);
% L:=SubgroupLattice(Gal);L;
% FixedField(K, L[12]);
% // 10
% //g:=x^6 + 84*x^5 + 34650*x^4 + 1830640*x^3 + 103984125*x^2 + 2202549804*x + 45380854071;
% // 9
% //g:=x^6 + 84*x^5 + 17094*x^4 + 847504*x^3 + 41527941*x^2 + 839165964*x + 7896111419;
% //11
%  //g:= x^6 - 30*x^5 - 1480*x^4 + 34600*x^3 - 245600*x^2 + 676000*x - 568000;
% //12
% g:=x^6 + 19194*x^4 + 141577317*x^2 + 244209993727;
% K:=SplittingField(g);
% GaloisGroup(K);

As a consequence of the previous results, we conclude:

\begin{thm}\label{thm:cond}
Let $f(x)= x^3+Bx+C\in\Z[x]$ be an irreducible polynomial such that all of the following conditions hold
\begin{enumerate}[(i)]
\item $2\nmid C$,
\item $3\nmid B$,
\item $7\mid B$ and $7\mid\mid C$,
\item $p\mid B$ and $p\mid \mid C$ for any prime $p\neq 2,3,7$ dividing $-4B^3-27C^2$.
\item The splitting field $\Q_f\not\simeq\Q(\zeta_7)$ nor the field $6.0.214375.1$. 
\end{enumerate}
Then $f$ gives a twist of the Klein quartic (constructed as in Subsection \ref{sec:good}) that is a counterexample to the Hasse principle.
\label{thmCase38}
\end{thm}

\begin{remark} Notice that the condition (iii) in Theorem \ref{thm:cond} implies that $f$ is irreducible because it is Eisenstein at $7$.
\end{remark}

In order to get polynomials $f$ satisfying all the conditions above we need $B=\pm2^a7^b\prod p_i^{e_i}$ and $C=\pm 3^c7\prod p_i$ with $e_i,b\geq 1$, $a,c\geq0$, the $p_i\neq 2, 3, 7$ prime and such that the discriminant is only divisible by the primes $7$ and $p_i$. This last condition yields an equality of the form: $3^{3+2c}\pm 1=2^{3a+2}7^{3b-2}p_i^{3e_i-2}$. In order to get the left hand size of this equality to be $0$ modulo $4$ we need the first sign to be positive, and in order to be $0$ modulo $7$ we need $3\mid c$. Actually, if $c=3c'$ with $c'\neq 3\text{ mod }7$ then $4\mid\mid3^{3+6c'}+1$ and $7\mid\mid 3^{3+6c'}+1$.

\begin{con} There are infinitely many values of $c'\neq 3\text{ mod }7$ such that all the primes different from $2$ and $7$ in the factorization of $3^{3+6c'}+1$ appear with an exponent  that is $1\text{ mod }3$.
\end{con}

% Factorization(3^(3+2*3)+1); 
% Factorization(3^(3+2*6)+1);
% Factorization(3^(3+2*12)+1);
% Factorization(3^(3+2*15)+1);
% Factorization(3^(3+2*18)+1);
% Factorization(3^(3+2*9)+1);

% //for i:=11 to 20 do 
% //    Factorization(3^(3+2*i)+1);
% //    Factorization(3^(3+2*i)-1);
% //end for;

Any such $c'$ produces a factorization $3^{3+6c'}+1=2^{2}7p_i^{3e_i-2}$ giving $B=-7\prod p_i^{e_i}$ and $C=\pm 3^{3c'}7\prod p_i$ satisfying conditions $(i)-(iv)$ in Theorem \ref{thm:cond}.

\begin{exm}\label{ex} After checking condition $(v)$ in Theorem \ref{thm:cond}, we find that the polynomials  $x^3-7x\pm7$, $x^3-7\cdot19\cdot37x\pm3^3\cdot7\cdot19\cdot37$ and  $x^3-7\cdot31\cdot61\cdot271x\pm3^6\cdot7\cdot31\cdot61\cdot271$  yield different twists of the Klein quartic that are counterexamples to the Hasse principle.
\end{exm}

Accordingly to the folklore conjecture that for curves, under the assumption that $Sha(J)$ is finite, the Brauer-Manin obstruction is the only
obstruction to the Hasse principle, we should have that the obstruction to have rational points of the counter-examples given in Example \ref{ex} is due to this phenomenon. If we could find a degree 1 rational divisor in their jacobian, we could indeed prove it under the finite $Sha$ assumption \cite[p. 37]{Scharaschkin}. Unfortunately, we were not able to find such a divisor.

\bibliographystyle{alpha}
\bibliography{Paper}

\begin{thebibliography}{LLLGR21}

\bibitem[BCP97]{magma}
Wieb Bosma, John Cannon, and Catherine Playoust.
\newblock The {M}agma algebra system. {I}. {T}he user language.
\newblock {\em J. Symbolic Comput.}, 24(3-4):235--265, 1997.
\newblock Computational algebra and number theory (London, 1993).

\bibitem[Elk98]{Elkies98}
Noam Elkies.
\newblock The {K}lein quartic in {N}umber {T}heory.
\newblock {\em The Eightfold Way: The Beauty of Klein's Quartic Curve}, 35,
  1998.

\bibitem[FLR05]{JFern}
J.~{Fern\'{a}ndez}, J.~C. {Lario}, and A.~{Rio}.
\newblock On twists of the modular curves {$X(p)$}.
\newblock {\em Bulletin of the London Mathematical Society}, 37(3):342--350,
  2005.

\bibitem[HK03]{halberstadt2003}
E.~{Halberstadt} and A.~{Kraus}.
\newblock Sur la courbe modulaire {$X_E(7)$}.
\newblock {\em Experiment. Math.}, 12(1):27--40, 2003.

\bibitem[LLLGR21]{Elisa-Ritzenthaler}
Reynald Lercier, Qing Liu, Elisa Lorenzo~Garc\'{\i}a, and Christophe
  Ritzenthaler.
\newblock Reduction type of smooth plane quartics.
\newblock {\em Algebra Number Theory}, 15(6):1429--1468, 2021.

\bibitem[{LMF}22]{lmfdb}
The {LMFDB Collaboration}.
\newblock The {L}-functions and modular forms database.
\newblock \url{http://www.lmfdb.org}, 2022.
\newblock [Online; accessed 17 November 2022].

\bibitem[{Lor}17]{Elisa1}
E.~{Lorenzo Garc\'{i}a}.
\newblock Twists of non-hyperelliptic curves.
\newblock {\em Revista Matem\'{a}tica Iberoamericana}, 33(1):169--182, 2017.

\bibitem[{Lor}18]{Elisa2}
E.~{Lorenzo Garc\'{i}a}.
\newblock Twists of genus 3 non-hyperelliptic curves.
\newblock {\em International Journal of Number Theory}, 14(6):1785--1812, 2018.

\bibitem[MT10]{Meagher2010}
S.~{Meagher} and J.~{Top}.
\newblock Twists of genus three curves over finite fields.
\newblock {\em Finite fields and their applications}, 16(5):347--368, 9 2010.

\bibitem[{Ozm}12]{2009arXiv0911.4537O}
E.~{Ozman}.
\newblock Local points on quadratic twists of {$X_0(N)$}.
\newblock {\em Acta Arithmetica}, 152(4):323--348, 2012.

\bibitem[{Ozm}13]{2012arXiv1205.3424O}
E.~{Ozman}.
\newblock On polyquadratic twists of {$X_0(N)$}.
\newblock {\em Journal of Number Theory}, 133:3325--3338, 2013.

\bibitem[PSS07]{Stoll2005}
B.~{Poonen}, E.~F. {Schaefer}, and M.~{Stoll}.
\newblock Twists of {$X(7)$} and primitive solutions to $x^2+y^3=z^7$.
\newblock {\em Duke Math. J.}, 133:103--158, 2007.

\bibitem[Sch99]{Scharaschkin}
Victor Scharaschkin.
\newblock {\em Local-global problems and the {B}rauer-{M}anin obstruction}.
\newblock ProQuest LLC, Ann Arbor, MI, 1999.
\newblock Thesis (Ph.D.)--University of Michigan.

\bibitem[{Zyw}15]{Zywina2015}
D.~{Zywina}.
\newblock {On the possible images of the mod $\ell$ representations associated
  to elliptic curves over $\Q$}.
\newblock {\em ArXiv e-prints}, August 2015.

\end{thebibliography}
%\printbibliography

\end{document}